\documentclass[11pt]{article}
\usepackage{amsthm}
\usepackage[mathscr]{eucal}
\usepackage{graphics}
\usepackage{endnotes}

\newcommand{\1}{{\mathbf{1}}}

\def\Om{\Omega}

\def\b{\beta}

\def\s{\sigma}
\def\m{\mu}

\def\La{\Lambda}

\usepackage{epsfig}

\newtheorem{theorem}{Theorem}[section]
\newtheorem{lemma}{Lemma}[section]
\newtheorem{proposition}{Proposition}[section]

\newtheorem{remark}{Remark}

\input{epsf}

\title{System of recursive equations for the partition functions of 1D models}
\author{U.A. Rozikov\textsuperscript{1}\\
{\small \textsuperscript{1} Institute of Mathematics and
Information Technologies,\hfill{\ }}\\
\ \ {\small 29, Do'rmon Yo'li str., Tashkent, 100125, Uzbekistan.}\\ \ \ \ \ \ {\small Email: rozikovu@yandex.ru \hfill {\ }}}

\date{}

\begin{document}

\maketitle

\begin{abstract}  In this note we consider several kind of
partition functions of one-dimensional models with nearest
- neighbor interactions $I_n, n\in \mathbf{Z}$ and  spin values $\pm 1$.
We derive systems of recursive equations for each kind of
such functions. These systems depend on parameters $I_n, n\in \mathbf{Z}$.
Under some conditions on the parameters
we describe solutions of the systems of recursive equations.

\end{abstract}

\section{Introduction}

The properties of physical, biological and many other systems
can be described by differential and recursive equations; the
latter are also called discrete dynamical system (see e.g.
\cite{De}, \cite{GR}, \cite{Sh}). Also systems of non-linear, higher dimensional
recursive equations arise in solving many different problems
(see e.g. \cite{D'}, \cite{Kr}, \cite{Mu}, \cite{Roz},\cite{SR}, \cite{Ta}).
But theory of the systems of recursive equations is not developed enough.
So for each such a system one has to use a specific argument which is suitable
for soling the system.

In the paper we consider the Hamiltonian (energy)
\begin{equation}
H(\s)= \sum_{l=(x-1, x): x\in \mathbf{Z}}
I_x\1_{\s(x-1)\ne\s(x)}, \label{1}
\end{equation}
 where $\mathbf{Z}=\{...,-2,-1, 0,
1,2,...\}$,\ $\s$ is a function (configuration), $\s:\mathbf{Z}\to \{-1,1\}$, (the set of all 
configurations $\s$ is denoted by $\Om=\{-1,1\}^{\mathbf {Z}})$
and $ I_x\in R$ for any $x\in \mathbf{Z}$.

In statistical physics the Hamiltonian (\ref{1}) is called an one-dimensional (1D) model.

Let us consider a sequence $\La_n=[-n,n]$, $n=0,1,...$ and denote
$\La_n^c=\mathbf{Z}\setminus \La_n$. Consider a boundary condition
$\s^{(+)}_n=\s_{\La_n^c}=\{\s(x)=+ 1: x\in \La_n^c\}.$ The energy
$H_n^+(\s)$ of the configuration $\s$ in the presence of the
boundary condition $\s^{(+)}_n$ is expressed by the formula
\begin{equation}
H^+_n(\s)= \sum_{l=(x-1, x): x\in \La_n}
I_x\1_{\s(x-1)\ne\s(x)}+
I_{-n}\1_{\s(-n)\ne 1}+I_{n+1}\1_{\s(n)\ne 1}. \label{2}
\end{equation}

The Gibbs measure on $\Om_n=\{-1, 1\}^{\La_n}$ with respect to the
boundary condition $\s_n^{(+)}$ is defined by 
\begin{equation}
\m^+_{n,\b}(\s)=Z^{-1}(n,\b,+)
\exp(-\b H_n^+(\s)), \label{3}
\end{equation}
where $\b=T^{-1}$, $T>0-$ temperature and
 $ Z(n, \b, +)$ is the normalizing factor 
 (statistical sum or partition function):
 \begin{equation}
 Z(n, \b, +)=\sum_{\varphi\in \Om_n}\exp(-\b H_n^+(\varphi)).\label{4}
 \end{equation}
 Note that the probability  
 (with respect to measure $\m^+_{n,\b}$) of a subset $\Om'_n$ of $\Om_n$
 is defined by 
\begin{equation}
 \m^+_{n,\b}(\Om'_n)=Z^{-1}(n,\b,+)
\sum_{\psi\in \Om'_n}\exp(-\b H_n^+(\psi))={Z'(n,\b,+)\over Z(n,\b,+)}, \label{5}
\end{equation}
where $Z'(n,\b,+)$ is called a "crystal" partition function:
\begin{equation}
Z'(n,\b,+)=\sum_{\psi\in \Om'_n}\exp(-\b H_n^+(\psi)).\label{6}
\end{equation}
 
So to define the Gibbs measure and probability of an event of the system 
one has to compute the partition functions.
If $\mu^+_\b=\lim_{n\to \infty}\m^+_{n,\b}$ exists then it is called a 
limit Gibbs measure.
A limit Gibbs measure for a given type of interaction (energy) 
may fail to be unique this means that the physical system with 
this interaction can take several distinct equilibria i.e there is phase transition.
  
Note that (see \cite{Ge}, p.95) for the model (\ref{1}) on $N=\{1,2,...\}$ it was
shown that there occurs a phase transition iff $\sum_{n\geq
1}e^{-2I_n}<\infty.$ In \cite{Ro} using a contour argument it has been
proven that for that model (\ref{1}) the phase transition occurs if
$I_n+I_{n+k}>k$ for any $n\in \mathbf{Z}, \ k\in N.$

In this paper we consider some (crystal) partition functions of the model
and give the system of recursive equations for the functions. Under some 
conditions on parameters of the model we describe their solutions.
 
\section{Partition function of "+" and "$\pm$" -boundary conditions}

Consider two type of partition functions:
\begin{equation}
Z^+_n=\sum_{\s_n\in \Om_n}\exp\{-\b H^+_n(\s_n)\}, \label{7}
\end{equation}
\begin{equation}
Z^{\pm}_n=\sum_{\s_n\in \Om_n}\exp\{-\b H^{\pm}_n(\s_n)\}, \label{8}
\end{equation}
where $H^+_n$ is defined by (\ref{2}) and
\begin{equation}
H^{\pm}_n(\s_n)=H^+_n(\s_n)+I_{-n}\s(-n). \label{9}
\end{equation}
In this paper for the simplicity assume
\begin{equation}
I_n=I_{-n+1},\ \ \mbox{for any} \ \ n\in \mathbf{Z}. \label{10}
\end{equation}
\begin{proposition}\label{p1}
If condition (\ref{10}) is satisfied then the
partition functions (\ref{7}) and (\ref{8}) have the form
\begin{equation} \begin{array}{llll}
Z_n^+={1\over 2}\bigg(\prod_{i=0}^n(1+e^{-\b I_{i+1}})^2+\prod_{i=0}^n(1-e^{-\b I_{i+1}})^2\bigg),\\
Z_n^{\pm}={1\over 2}\bigg(\prod_{i=0}^n(1+e^{-\b I_{i+1}})^2-\prod_{i=0}^n(1-e^{-\b I_{i+1}})^2\bigg).\\
\end{array}\label{11}
\end{equation}
\end{proposition}
\begin{proof}
Under the condition (\ref{10}) we get $Z^-_n=Z^+_n$ and $Z^{\pm}_n=Z^{\mp}_n$.
Now from (\ref{7}), (\ref{8}) we obtain the following system of recursive equations
\begin{equation}
\begin{array}{llll}
Z^+_n=(1+e^{-2\b I_{n+1}})Z^+_{n-1}+2e^{-\b
I_{n+1}}Z^{\pm}_{n-1},\\
 Z^{\pm}_n=(1+e^{-2\b
I_{n+1}})Z^{\pm}_{n-1}+2e^{-\b I_{n+1}}Z^+_{n-1}
\end{array}\label{12}
\end{equation}
Putting $X_n=Z^+_n-Z^{\pm}_n$ and $Y_n=Z^+_n+Z^{\pm}_n$ from (\ref{12})
we get
\begin{equation}
\begin{array}{llll}
X_n=(1-e^{-\b I_{n+1}})^2X_{n-1},\\
Y_n=(1+e^{-\b I_{n+1}})^2Y_{n-1}.\\
\end{array}\label{13}
\end{equation}

The equalities $X_0=Z^+_0-Z^{\pm}_0=(1-e^{-\b I_1})^2, \ \
Y_0=(1+e^{-\b I_1})^2$  with (\ref{13}) imply
$$
X_n=\prod_{i=0}^n(1-e^{-\b I_{i+1}})^2, \ \
Y_n=\prod_{i=0}^n(1+e^{-\b I_{i+1}})^2.$$ Hence we get (\ref{11}).
\end{proof}

For example, in a case of the usual Ising model i.e. $I_m=I,$
$\forall m\in \mathbf{Z}$ from (\ref{11}) denoting $\tau=\exp(-\b I)$ we get
$$\begin{array}{llll}
Z_n^+={1\over 2}\bigg((1+\tau)^{2(n+1)}+(1-\tau)^{2(n+1)}\bigg),\\
Z_n^{\pm}={1\over 2}\bigg((1+\tau)^{2(n+1)}-(1-\tau)^{2(n+1)}\bigg).\\
\end{array}$$ Using these equalities (for usual Ising model) it is easy to see that
$$ {Z_n^+\over Z_n^{\pm}}\to 1, \ \ {\rm if} \ \ n\to\infty .$$
This means that for the Ising model the partition functions 
$Z_n^+$ and $Z^{\pm}_n$ are asymptotically equal. This gives
in fact uniqueness of limit Gibbs measure for the 1D Ising model. 
Such an asymptotical equality is true if $I_m$ is 
a periodic function of $m$ i.e $I_{m+p}=I_m$ for 
some $p\geq 1$ and all $m\in N$.

\section{Crystal partition functions}

In this section we are going to
describe the crystal partition functions.

Denote $\Om_{m,n}=\{-1,1\}^{[m,n]},$ where $[m,n]=\{m, m+1,...,
n\},\ m,n \in \mathbf{Z}, \ n\geq m.$ Put 
$$N_\varepsilon(\s)=|\{x\in [m,n]:
\s(x)=\varepsilon\}|, \ \varepsilon=\pm 1,$$
where $|S|$ is the cardinal  of the set $S$.
For $r=0,1,..., n-m+1$ consider the following 
crystal partition functions:
\begin{equation}
Z^{\varepsilon, r}_{m,n}=\sum_{\s\in\Om_{m,n}:N_{-\varepsilon}(\s)=r}e^{-\b
H^\varepsilon(\s)}, \ \ \varepsilon =-,+ \label{14}
\end{equation}
\begin{equation}
Z^{\pm, r}_{m,n}=\sum_{\s\in\Om_{m,n}:N_+(\s)=r}e^{-\b
H^\pm(\s)}. \label{15}
\end{equation}

 Note that $Z^{-,r}_{m,n}=Z^{+,r}_{m,n}$. Denoting $X^r_{m,n}=Z^{-,r}_{m,n}$
 and $Y^r_{m,n}=Z^{\pm, r}_{m,n}$ from (\ref{14}) and (\ref{15}) one easily 
 gets the following system of (multi-variable) recursive equations 
\begin{equation}
\begin{array}{llll}
X^r_{m,n}=X^{r-1}_{m,n-1}+e^{-\b
I_{n+1}}Y^r_{m, n-1},\\
 Y^r_{m,n}=Y^r_{m,n-1}+e^{-\b I_{n+1}}X^{r-1}_{m, n-1},
\end{array}\label{16}
\end{equation}
where $r=0,1,...,n-m+1$, $m,n\in \mathbf{Z}$, $n\geq m$.

The system (\ref{16}) can be reduced to a recursive equation 
with respect to $X^r_{m,n}$. Indeed, from the first equation of (\ref{16})
we get 
\begin{equation}
Y^r_{m, n-1}=
e^{\b I_{n+1}}\left(X^r_{m,n}-X^{r-1}_{m,n-1}\right).\label{17}
\end{equation}
Now from the second equation of (\ref{16}) using (\ref{17})
we get  
\begin{equation}
X^r_{m,n}=X^{r-1}_{m,n-1}+e^{-\b
(I_{n+1}-I_n)}(X^r_{m,n-1}-X^{r-1}_{m,n-2})+e^{-\b(I_n+I_{n+1})}X^{r-1}_{m,n-2},
\label{18}
\end{equation}
 where $r=1,2,...,n-m+1,\ n\geq m$ and
$$ X^0_{m,m}=1, \ \ X^1_{m,m}=e^{-\b (I_m+I_{m+1})}, \ \
X^0_{m,m+1}=1,$$
$$X^1_{m,m+1}=e^{-\b (I_m+I_{m+1})}+e^{-\b(I_{m+1}+I_{m+2})},
\ \ X^2_{m,m+1}=e^{-\b (I_m+I_{m+2})},\ m\in \mathbf{Z}.$$

Iterating (\ref{18}) one can obtain an expression for $X^r_{m,n}.$
Then using (\ref{17}) one can find $Y^r_{m,n}.$ But these
expressions would be in a very bulky form.

Now we shall illustrate such an expression for the Ising model i.e.
$I_m\equiv I,\ m\in \mathbf{Z}$. In this case the recurrence equation (\ref{18})
becomes more simple
\begin{equation}
X^r_n=X^{r-1}_{n-1}+X^r_{n-1}+(\chi-1)X^{r-1}_{n-2},\label{19}
\end{equation}
where $X^r_n=X^r_{m,n_1}$ with $n_1-m=n$, $\chi=e^{-2\b I}.$

 It is easy to see that for $r=0,1,2,3,4$ the solutions are
$$X^0_n=1, \ \ X^1_n=n\chi, \ \ X^2_n=(n-1)\chi+{(n-2)(n-1)\over
2!}\chi^2,$$
$$X^3_n=(n-2)\chi+(n-2)(n-3)\chi^2+{(n-2)(n-3)(n-4)\over
3!}\chi^3,$$
$$X^4_n=(n-3)\chi+{3(n-3)(n-4)\over 2}\chi^2+{(n-3)(n-4)(n-5)\over
2}\chi^3+$$
$${(n-3)(n-4)(n-5)(n-6)\over 4!}\chi^4,$$ here we used
the following formulas $$\sum_{j=1}^nj^2={1\over 6}(2n^3+3n^2+n),\ \
\sum_{j=1}^nj^3={1\over 4}(n^4+2n^3+n^2).$$

Note that $X^r_n$ has a form
$$X^r_n=a_{1,n}^r\chi+a_{2,n}^r\chi^2+...+a_{r,n}^r\chi^r.$$
 For the coefficients $a_{k,n}^r$, $0\leq k\leq r\leq n$, using (\ref{19})
we obtain the following system of recursive equations
\begin{equation}
\begin{array}{llll}
a^r_{1,n}=a^r_{1,n-1}+a^{r-1}_{1,n-1}-a^{r-1}_{1,n-2},\\[2mm]
a^r_{k,n}=a^r_{k,n-1}+a^{r-1}_{k,n-1}+a^{r-1}_{k-1,n-2}-a^{r-1}_{k,n-2},\
\
k=2,3,...,r-1,\\[2mm]
a^r_{r,n}=a^r_{r,n-1}+a^{r-1}_{r-1,n-2},\\
\end{array}\label{20}
\end{equation}
The following lemma gives solution to (\ref{20}) 
\begin{lemma}
Solution of the system of recursive equations (\ref{20}) is
\begin{equation}
a^r_{k,n}={n-r+1 \choose k}{r-1\choose k-1},
\ \ 0\leq k\leq r\leq n. \label{21}
\end{equation}
\end{lemma}
\begin{proof} We shall use mathematical induction (cf. with \cite{Ni}
pages 148-150). Let $A_m$ denote all cases of (\ref{21}) with $n+k+r=m.$
Formulas given above for $X^r_n,\ \ r=0,1,2,3,4$ show that the
formula (\ref{21}) is true for small values of $m$. Assuming that $A_m$
holds, we are to prove $A_{m+1}$ that is, equation (\ref{20}), for any
integers $n, r$ and $k$ whose sum is $m+1.$ Since RHS of equation
(\ref{20}) contains terms with $n+r+k\leq m$  using the assumption of
the induction for each term of RHS of (\ref{20}) we get (\ref{21}).
\end{proof}

Thus the solution of (\ref{19}) is given by
\begin{equation}
X^r_n=\sum_{k=1}^r{n-r+1 \choose k} {r-1\choose
k-1}\chi^k \label{22}
\end{equation}
\begin{remark}
Note that for $\chi=1$ (i.e. there is no
interaction) the solution of (\ref{19}) is $X^r_n={n\choose r}.$ Using
(\ref{22}) for $\chi=1$ we obtain the following property of binomial
coefficients
\begin{equation}
{n\choose r}=\sum_{k=1}^r{n-r+1 \choose k}{r-1\choose
k-1}. \label{23}
\end{equation}
 This identity is known as the convolution
identity of Vandermonde.
\end{remark}

Since interaction (parameter $I$) of the 1D Ising model is 
translation-invariant (does not depend on the points of $\mathbf{Z}$), 
the unknown functions  $X^r_{m,n}, \ Y^r_{m,n}$ of the system (\ref{16}) depend 
on $n-m$ and $r$ only (see (\ref{19})), consequently, 
instate of $n-m$ we can write $n$. 
Summarizing the results for the Ising model we have
\begin{theorem}
For the Ising model the solution of the system of recursive 
equations (\ref{16}) is
$$X^r_n=\sum_{k=1}^r{n-r+1 \choose k} {r-1\choose
k-1}\chi^k,$$
$$Y^r_n={1\over \sqrt{\chi}}\left(X^r_{n+1}-X^{r-1}_n\right),$$
where $n$ stands for $n-m$. 
\end{theorem}   

{\bf Acknowledgments.} I thank the Abdus Salam International Center for
Theoretical Physics (ICTP), Trieste, Italy for providing financial
support of my visit to ICTP (February-April 2009).

\end{document}